\documentclass[11pt]{amsart}

\usepackage[backend=biber]{biblatex}
\usepackage{graphicx}%
\usepackage{multirow}%
\usepackage{amsmath,amssymb,amsfonts}%
\usepackage{amsthm}%
\usepackage{mathrsfs}%
\usepackage[title]{appendix}%
\usepackage{xcolor}%
\usepackage{textcomp}%
\usepackage{manyfoot}%
\usepackage{booktabs}%
\usepackage{algorithm}%
\usepackage{algorithmicx}%
\usepackage{algpseudocode}%
\usepackage{listings}%
\usepackage{mathtools}
\usepackage{setspace}
\usepackage{enumitem}

\newtheorem{theorem}{Theorem}[section]
\newtheorem{proposition}[theorem]{Proposition}
\newtheorem{lemma}[theorem]{Lemma}
\newtheorem{corollary}[theorem]{Corollary}

\theoremstyle{definition}
\newtheorem{definition}[theorem]{Definition}

\theoremstyle{remark}
\newtheorem{remark}[theorem]{Remark}
\newtheorem{example}[theorem]{Example}

\newcommand{\Z}{\ensuremath{\mathbb{Z}}}
\newcommand{\F}{\ensuremath{\mathbb{F}}}

\newcommand{\deff}[1]{\textbf{\emph{#1}}}

\newcommand{\hyph}{\mbox{-}}
\newcommand{\symdif}{\mathbin{\bigtriangleup}}

\DeclareMathOperator{\aff}{aff}
\DeclareMathOperator{\qc}{QC}
\DeclareMathOperator{\AG}{AG}

\title[How Many Cards Should You Lay Out in Quad-128]{How Many Cards Should You Lay Out in Quad-128: A Classification of Caps in $\AG(7,2)$}

\author[Calta]{Kariane Calta}
\address{Department of Mathematics \& Statistics, Vassar College, Poughkeepsie, New York 12604}
\email{kacalta@vassar.edu}

\author[Goldberg]{Timothy E.~Goldberg}
\address{Department of Mathematics, Lenoir-Rhyne University, Hickory, North Carolina 28601}
\email{timothy.goldberg@lr.edu}

\author[Rose]{Lauren L.~Rose}
\address{Department of Mathematics, Bard College, Annandale-on-Hudson 12504}
\email{rose@bard.edu}

\date{\today}

\keywords{finite geometry, affine geometry, recreational mathematics, combinatorics, Sidon sets, caps, EvenQuads}

\subjclass{05B25, 51E10, 51E15}

\begin{document}

\begin{abstract}
We define a cap in the affine geometry $\AG(n,2)$ to be a subset in which every collection of four points is in general position. In this paper, we classify, up to affine equivalence, all caps in $\AG(7,2)$ of size $k \geq 10$.  In particular, we show that there are two equivalence classes of $10$-caps and one equivalence class of $11$-caps, none of which are complete, and one equivalence class of $12$-caps, which are both complete and of maximum size.
\end{abstract}

\maketitle

\tableofcontents

\section{Introduction}

\emph{Quads} is a pattern recognition game similar to the popular card game SET. A standard Quads deck consists of $64$ cards, each with one, two, three, or four of the same symbol, in one of four colors. A complete Quads deck is shown in Figure~\ref{fig:quad64}, with triangles, squares, circles, and hearts as symbols. The goal of the game is to find \textbf{quads}, collections of four cards that satisfy a specific pattern, namely that in each attribute (number, symbol, shape) the four cards are all the same, all different, or two pairs of types. 

For example, the cards \emph{1-red triangle}, \emph{2-green squares}, \emph{3-red hearts}, and \emph{4-blue circles} form a quad, since the numbers are all different, there are two red and two green cards, and the shapes are all different.
\begin{figure}[ht]
    \centering
    \includegraphics[width=0.4\linewidth]{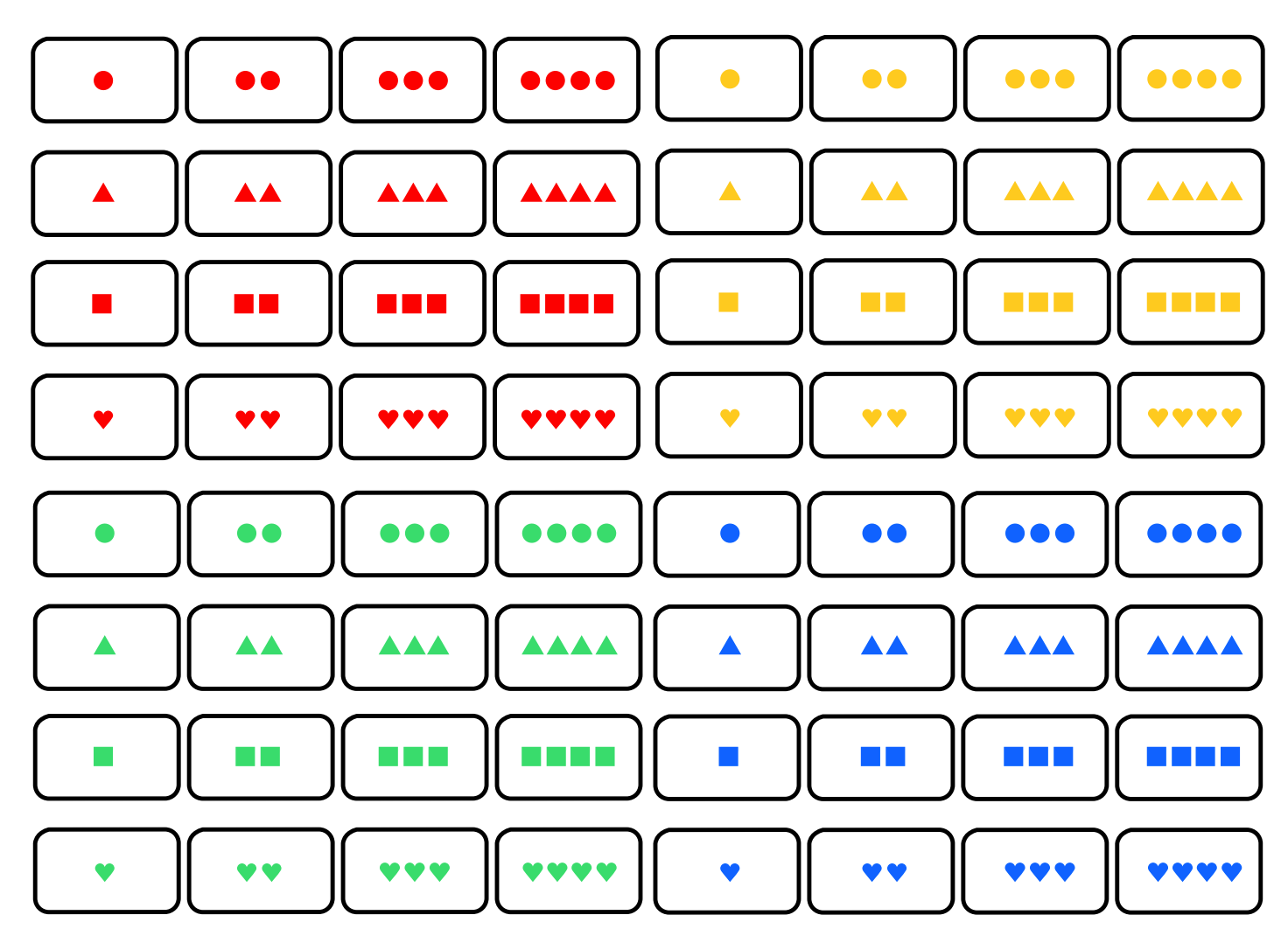}
    \caption{Quad-64 Deck}
    \label{fig:quad64}
\end{figure}
Quads cards can be coded with vectors in $\Z_2^6$, and easily generalized to $\Z_2^n$.  See \cite{lamatpaper} or \cite{rose-chapter} for details.

Once we do this, a quad has a nice mathematical interpretation.

\bigskip

\begin{theorem}\cite{lamatpaper}
Four distinct cards $A, B, C, D$ form a quad if and only if as vectors, $A + B + C + D = \mathbf{0}$.
\end{theorem}

\bigskip

This is a parallel to SET, where cards can be coded with vectors in $\Z_3^4$, and generalized to $\Z_3^n$. In this case, three distinct cards form a SET if and only if their vector sum is zero.

The celebrated \deff{Cap Set Problem} is to find the maximum size of a SET-free collection of cards, called a \deff{cap}. The answer is only known for $n \leq 6$, and the best known upper bound is $(2.756)^n$, due to Ellenberg and Gijswijt (\cite{ellenberg}). See \cite{maclagan} or \cite[Chapter 9]{Rosemcmahon2016joy} for a detailed introduction to the \emph{Cap Set} problem. 

In this paper, we study the analogous problem for Quads, but expand our study to the classification of all caps, not just those of maximum size. Quad-free sets are called \deff{$2$-caps} in \cite{Bennett} and \deff{Sidon Sets} in \cite{taitwon}, \cite{codes}, and \cite{nagy}, for example. The latter two papers make use of a natural connection between Sidon Sets and linear codes. In a slight abuse of definition, we call a quad-free set of cards a \deff{cap}. 

In \cite{taitwon}, Tait and Won found asymptotic bounds for the maximum size $M(n)$ of a cap in $\Z_2^n$:
\[
    \frac{1}{\sqrt{2}} \, 2^{n/2}
    \leq M(n)
    \leq 1 + \sqrt{2} \cdot 2^{n/2}.
\]
A \deff{complete cap} is a maximal quad-free set, but not necessarily of  maximum size. In \cite{RedmanRoseWalker}, Redman, Rose, and Walker constructed a small complete cap in $\Z_2^n$ for each $n$. 

The sizes of complete caps in dimensions $7$ and $8$ can be verified computationally, and the maximum caps sizes for dimensions $9$ and $10$ have been determined using libraries of even linear codes \cite{nagy}. 


In \cite{lamatpaper}, Crager et al.~classified caps of up to $9$ cards in any dimension, and all caps in  $\AG(n,2)$ for $n \leq 6$, up to affine equivalence, using the fact that the encoding by vectors in $\Z_2^n$ is a model for the affine geometry $\AG(n,2)$.


This paper picks up where \cite{lamatpaper} leaves off. Our goal is to classify caps in $\Z_2^7$, which corresponds to the card deck Quad-128. One way to construct Quad-128 from the standard deck, Quad-64, is to add a half-attribute, say background color, with two states instead of four, as seen in Figure~\ref{fig:q128}.
\begin{figure}[ht]
    \centering
    \includegraphics[width=4.13cm]{quad64.png}
    \includegraphics[width=4cm]{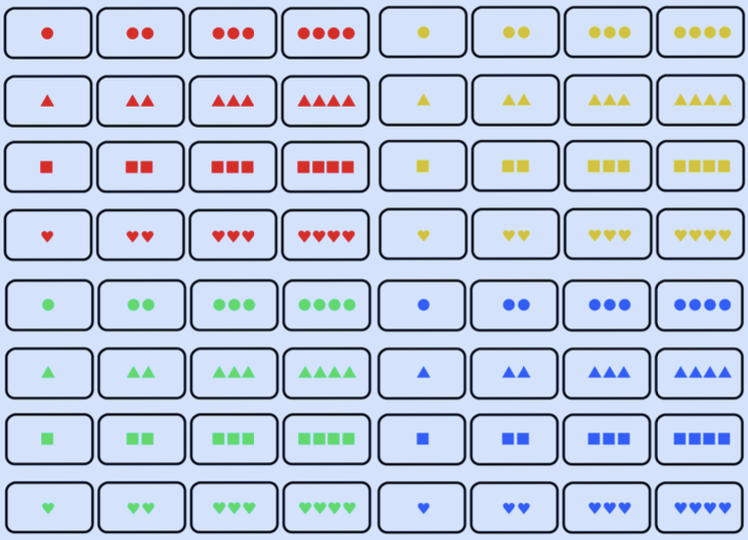}
    \caption{A Quad-128 deck}
    \label{fig:q128}
\end{figure}

Since caps in $\Z_2^6$ can be embedded as caps in $\Z_2^7$, the results from \cite{lamatpaper} describe all caps of size up to $9$ in $\Z_2^7$. Our focus in this paper is on caps of size $10$ or higher. Our primary result is a classification of these caps into affine equivalence classes, along with explicit algebraic and combinatorial constructions based on the affine dependence relations among cap elements. 

\bigskip

\begin{theorem}
    In $\AG(7,2)$:
    \begin{enumerate}
        \item There are two affine equivalence classes of caps of size $10$.
        \item There is one affine equivalence class each of caps of sizes $11$ and $12$.
    \end{enumerate}
\end{theorem}

\bigskip

We also have the following corollary.

\bigskip

\begin{corollary} \leavevmode
    \begin{enumerate}
        \item The maximum cap size in $\AG(7,2)$ is $12$.
        \item A cap in $\AG(7,2)$ is complete if and only if it has size $12$.
        \item No cap of size $10$ or $11$ is complete in  $\AG(n,2)$, for any $n$.
    \end{enumerate}
\end{corollary}

\bigskip

\begin{remark}
In terms of the card game Quad-128, these results mean that you must lay out $13$ or more cards to ensure the presence of a quad. 
\end{remark}


\section[Fundamental notions in AG(n,2)]{Fundamental notions in $\AG(n,2)$}

The set $\AG(n,q)$ is the finite affine geometry of dimension $n$ and order $q$. A standard model for $\AG(n,q)$ consists of vectors in $\F_q^n$, where $\F_q$ is the field of order $q$. In this model the affine subspaces (called \emph{flats}) are just translates of linear subspaces. A detailed description of the geometry of $\AG(n,q)$, and in particular $\AG(n,2)$, can be found in \cite{lamatpaper}. For the convenience of the reader, we recall the most relevant definitions and properties here. Let $V$ be a vector space over a field.

An \deff{affine combination} of a set $A \subseteq V$ is a linear combination of elements of $A$ whose coefficients sum to $1$. The set $A$ is \deff{affinely independent} if no element is an affine combination of the others. The \deff{affine span} of $A$, denoted $\aff(A)$, consists of all affine combinations of elements of $A$. The \deff{flats} in $V$ are the affine spans of non-empty sets. An \deff{affine basis} $B$ for a flat $F$ is a set of affinely independent elements whose (affine) span is $F$. The \deff{dimension} of $F$ is $\dim(F) = |B| - 1$. Equivalently, if $F = W + v$, where $W$ is a linear subspace and $v \in V$, then $\dim(F) = \dim(W)$.

An \deff{affine transformation} $T : V \rightarrow  W$ is a map of the form $T(x) = L(x) + b$ where $L$ is a linear map and $b \in W$ is fixed. Affine transformations map affinely dependent sets to affinely dependent sets.
Two subsets are \deff{affinely equivalent} if there is an affine isomorphism (invertible affine transformation) that maps one onto the other.

In the case of $\Z_2^n$, the notion of an affine combination becomes much simpler. 

\bigskip

\begin{lemma} \label{lem:odd-sum}
    An affine combination of a set $S$ in $\Z_2^n$ is exactly a sum of an odd number of elements of $S$.
\end{lemma}

\begin{proof}
    This follows immediately from the fact that the only scalars are the elements of $\Z_2 = \{0,1\}$, and a sum of these scalars equals $1$ if and only if there is an odd number of ones in this sum.    
\end{proof}

It follows from the lemma above that the affine span of a set $S$ in $\Z_2^n$ is the minimal superset of the $S$ that is closed under sums of odd numbers of elements.

\bigskip

\begin{definition}
    Let $S \subseteq \Z_2^n$. An \deff{affine dependence relation} of elements of $S$ is an expression of one element as an affine combination (i.e.~sum of an odd number) of other elements, or any equation between sums of elements that can be rearranged into this form.
\end{definition}

\bigskip

\begin{remark} \label{preserved-dependencies}
Since affine transformations map affinely dependent sets to affinely dependent sets, it follows that they preserve affine dependence relations. Hence if $S_1$ and $S_2$ are affinely equivalent via some affine isomorphism $f$, then $f$ maps the sets of all affine dependence relations in $S_1$ and $S_2$ to each other in a one-to-one fashion.
\end{remark}




\bigskip

\begin{definition}
    Let $S \subseteq \Z_2^n$. 
    \begin{enumerate}
        \item The \deff{dimension} of $S$ is the dimension of $\aff(S)$. 
        \item A \deff{basis} for $S$ is a subset of $S$ that forms an affine basis for $\aff(S)$.
        \item The \deff{dependent} set in $S$ with respect to $B$ is $D = S \setminus B$.
    \end{enumerate}
\end{definition}

\bigskip

\begin{remark}
    Note that if $S$ is an $m$-dimensional subset with affine basis $B$, then $|B|=m+1$; i.e.~an affine basis has one more element than the dimension of the subset.
\end{remark}

We are primarily interested in the case where $S$ is a $7$-dimensional subset $\Z_2^7$ of size $k$ where  $k \geq 10$. In this case $|B|=8$, and so $|D| = k - 8 \geq 2$. 

\bigskip

\begin{example}
Let $S$ be a set of size $10$ with basis $B= \{a_1, \ldots, a_8\}$ and $D=\{x, y\}$. Then $x$ and $y$ can each be written as a sum of an odd number of basis elements. Since $x, y \notin B$ and $|B| = 8$, each must be the sum of either three, five, or seven basis elements. 
\end{example}

\bigskip

\begin{definition}
    Let $S \subseteq \Z_2^n$ be a subset with basis $B$ and dependent set $D = \{x_1, \ldots x_r\}$.
    \begin{enumerate}
        \item For $1 \leq i \leq r$, the set $B_i$ is the unique set of basis elements that sum to $x_i$.
        \item For any $1 \leq i_j \leq r$, we define $B_{i_1 i_2 \cdots i_s}= B_{i_1} \cap B_{i_2} \cdots \cap B_{i_s}$.
        \item More generally, for any $x \in D$ we define $B_x$ to be the unique set of basis elements that sum to $x$. 
    \end{enumerate}
\end{definition}

\bigskip

Using the inclusion-exclusion principle, we can determine $|B_{i_1 i_2 \cdots i_s}|$ for any subset $\{i_1, i_2, \ldots, i_s\}$ of $\{1, \ldots, n\}$ with size $s$ if we know the cardinalities of the every $B_{i_1 \ldots i_s}$. Some of this numerical information is captured by the \emph{type} and \emph{extended type} of a set, defined below.

\bigskip

\begin{definition} 
    Let $S \subseteq \Z_2^n$ be a set with basis $B=\{b_1, \ldots, b_m\}$ and dependent set $D = \{x_1, \ldots, x_r\}$. We say $B$ has \deff{type}
    \[
        |B_1| \hyph |B_2| \hyph \cdots \hyph |B_m|,
    \]
    where the cardinalities are listed in non-increasing order, and \deff{extended type}
    \[
        |B_1| \hyph |B_2| \hyph \cdots \hyph |B_r| \hyph (|B_{12}|, \ldots , |B_{1r}|, \ldots, |B_{(r-1)r}|),
    \]
    where the pairs $\{i,j\}$ are in lexicographic order. (This notation uses hyphenated numbers, not subtraction.)
\end{definition}

\bigskip

\begin{example} \label{type-example}
    Let $S$ be a set with basis $B= \{a_1, \ldots, a_n\}$ and $D=\{x_1, x_2\}$. Suppose $x_1 = a_1 + \cdots + a_5$ and $x_2 = a_3 + \cdots + a_7$. Then $B_1 = \{a_1, \ldots, a_5\}$, $B_2 = \{a_3, \ldots, a_7\}$, and $B_{12} = \{a_3, a_4, a_5\}$.
    Then $S$ with basis $B$ has type $5 \hyph 5$ and extended type $5 \hyph 5 \hyph (3)$, since $D$ contains two elements, $x_1$ and $x_2$, each of which are sums of five basis elements, and these sums have three elements in common.
\end{example}

\bigskip

\begin{remark} \label{rem:affine-equivalence}
    Note that if $S$ and $S'$ are subsets that are affinely equivalent via a function $f$, then for any basis $B$ for $S$, the image $f(B)$ is a basis for $S'$ and these two bases have the same extended type.
    Note that $S'$ may have a basis $B'$ of a different type from $B$, but in this case $f(B)$ cannot equal $B'$.
\end{remark}

\bigskip

\begin{example} \label{template-example}
    Suppose $S_1$ and $S_2$ are sets with bases $B_1 = \{a_1, \ldots, a_n\}$ and $B_2 = \{b_1, \ldots, b_n\}$, respectively, and dependent sets $D_1 = \{x_1, x_2\}$ and $D_2 = \{y_1, y_2\}$, respectively.
    Suppose further that
    \begin{align*}
        x_1 &= a_1 + \cdots + a_5, \\
        y_1 &= b_1 + \cdots + b_5
    \end{align*}
    and
    \begin{align*}
        x_2 &= a_3 + \cdots + a_7, \\
        y_2 &= b_3 + \cdots + b_7.
    \end{align*}
    Then we can build an affine isomorphism $f$ between $\aff(S_1) = \aff(B_1)$ and $\aff(S_2) = \aff(B_2)$ by setting $f(a_i) = b_i$ for each $a_i \in B_1$, and extending affine linearly. The function $f$ must then preserve affine combinations, so
    \begin{align*}
        f(x_1)
        &= f(a_1 + \cdots + a_5) \\
        &= f(a_1) + \cdots + f(a_5) \\
        &= b_1 + \cdots + b_5 \\
        &= y_1,
    \end{align*}
    and similarly $f(x_2) = y_2$. Hence $S_1$ and $S_2$ are affinely equivalent via $f$.
\end{example}

The situation described in Example~\ref{template-example} will occur multiple times in this paper. We will often prove that, for a set with a basis of a certain type, we can always label the points so that the dependent elements have particular expressions as sums of basis elements. We will call this a \deff{dependent set template} with respect to the basis. By an argument like that in the example above, we have the following lemma.

\bigskip

\begin{lemma} \label{same-template}
    If two subsets of $\Z_2^n$ have the same size and dimension, and have bases such that their dependent sets fit the same template, then the sets are affinely equivalent.
\end{lemma}

\bigskip

\begin{remark}
    A surprising consequence of the results in this paper is that, for $7$-dimensional caps in $\Z_2^n$, two caps are affinely equivalent if and only if they possess bases of the same extended type.
    This is not true, however, for higher-dimensional caps, as seen in Example~\ref{higherdim}.
\end{remark}

\bigskip

\begin{example} \label{higherdim}
    Let $C = B \cup D$ where $B=\{a_1, \ldots a_9\}$ is an independent set and $D = \{x, y, z\}$, and
    \begin{align*}
        x &= a_1 + \cdots + a_5, \\
        y &= a_1 + a_2 + a_3 + a_6 + a_7, \\
        z &= a_1 + a_2 + a_3 + a_8 + a_9,
    \end{align*}
    Let $C'= B \cup D'$ where $D' = \{u, v, w\}$ and
    \begin{align*}
        u &= a_1 + \cdots + a_5, \\
        v &= a_1 + a_2 + a_3 + a_6 + a_7, \\
        w &= a_1 + a_2 + a_4 + a_6 + a_8.
    \end{align*}
    
    Both $C$ and $C'$ of these have $B$ of extended type $5 \hyph 5 \hyph 5 \hyph (3,3,3)$, but $C$ is not equivalent to $C'$ for the following reason. In $C$, all basis elements appear in at least one dependence relation. By Remark~\ref{preserved-dependencies}, this property must be preserved by affine equivalence, but in $C'$, the basis element $a_9$ does not appear in any dependence relations.
\end{example}

If $B$ and $B'$ are bases for the same set, they may have different types, but it is always the case that $|B'|=|B|$ and $|D'|=|D|$. We will use the next theorem throughout this paper to change from a basis of one type to another basis of a different type. 

\bigskip

\begin{theorem}[Affine Basis Exchange Theorem] \label{newcapbasis}
    Let $S \subseteq \Z_2^n$ be a set with basis $B$ and dependent set $D =\{x_1, x_2, \ldots, x_r\}$.
    Suppose that there exists $a  \in (B_1 \cap B_2) \setminus (\cup_{i=3}^r B_i)$. 
    Let $B' = (B \setminus \{a\}) \cup \{x_2\}$ and $D' = (D \setminus \{ x_2 \}) \cup \{a \}$. Then:
    \begin{enumerate}
        \item $B'$ is a basis for $S$ with dependent set $D'$.
        \item $B'_{1} = (B_{1} \symdif B_{2}) \cup \{x_2\}$, where $\symdif$ denotes the symmetric difference operator.
        \item $B'_{a}$ = $B_2 \setminus \{a\}) \cup \{x_2\}$.
        \item $B'_{i} = B_{i}$ for $i>2$.
    \end{enumerate}
\end{theorem}

\begin{proof} 
    Let $S$, $B$, $D$, $a$, $B'$, and $D'$ be as in the statement of the theorem.
    
    First note that $B'$ is also a basis for $S$ because the elements of $B \setminus \{a\}$ are independent by definition, and $x_2$ is independent of the elements of $B \setminus \{a\}$ since the only sum involving $x_2$ also contains $a$. 

    We now give an explicit construction of the elements of $D'$ in terms of elements of $B'$.
    First, since $a$ only appears in $B_1$ and $B_2$, each $x_i$ for $i > 2$ has the same sum of elements of $B'$ as it does for the elements of $B$. Thus, $B'_{i} = B_{i}$ for $i>2$. 
    Next, since $a \in B_{12}$, we have 
    \[
        x_2 = a + c_1 + \cdots + c_m
    \]
    where $\{c_1, \ldots, c_m\} = B_2 \setminus \{a\}$.
    Solving for $a$, we get
    \[
        a = x_2 + c_1 + \cdots + c_m.
    \]
    We have now written $a$ as a sum of elements in $B \setminus \{a\}$. We have essentially swapped the roles of $a$ and $x_2$, i.e.~$x_2 \in B'$ and $a$ is now a dependent element with respect to $B'$. Thus $B'_2 = (B_2 \setminus \{a\}) \cup \{x_2\}$. 

    Now, we cannot set $B'_1 = B_1$ since $a \notin B'$. Instead, we use the equation formed by adding $x_1$ and $x_2$. Since we are adding modulo $2$, the basis elements in $B_{12}$ cancel and so $x_1 + x_2$ is the sum of the basis elements in $B_1 \symdif B_2$, the symmetric difference of $B_1$ and $B_2$.  

    On the other hand, since $a \in B_{12}$ we have $a \notin B_1 \symdif B_2$. Thus, $a$ cannot appear as an addend in the sum $x_1 + x_2$. So we have that  
    \[
        x_1 + x_2 = s_1 + \cdots + s_{\ell},
    \]
    where $\{s_1, \ldots, s_{\ell}\} = B \symdif B'$. 
    Solving for $x_1$ we get
    \[
        x_1 = x_2 + s_1 + \ldots + s_{\ell}.
    \] 
    It follows that
    \[
        B'_{1} = (B_{1} \symdif B_{2}) \cup \{x_2\}.
    \]

    Finally, for $i>2$ the set $B_i$ does not contain $a$ or $x_2$, so it follows that the unique expression for $x_i$ in terms of the basis $B$ is the same as that for the basis $B'$, and hence $B'_i = B_i$.
\end{proof}

\bigskip

\begin{example} 
    Let $S= \{a_1, \ldots, a_8, x, y, z \}$ be a subset of $\Z_2^n$ with basis $B = \{a_1, \ldots, a_8\}$ and dependent set $D = \{x, y, z\}$, and suppose that
    \begin{align*}
        x &= a_1 + a_2 + a_3 + a_4 + a_5 + a_6 + a_7, \\
        y &= a_1 + a_2 + a_3+ a_4 + a_8, \\
        z &= a_1 + a_2 + a_5 + a_6 + a_8.
    \end{align*} 
    Then
    \[
        x + y = a_5 + a_6 + a_7 + a_8.
    \]
    Notice that $a_3 \in B_{xy} \setminus B_z$, so we can use the \emph{Affine Basis Exchange Theorem} to write $S = B' \cup D'$, where $B'= (B \setminus \{a_3\}) \cup \{y\}$ is a basis for $C$, and $D' = \{ a_3, x, z \}$.
    Rewriting the equations for $y$ and $x+y$, we get  
    \begin{align*}
        a_3 &= a_1 + a_2 + a_4 + a_8 + y,  \\ 
        x &= a_5 + a_6 + a_7 + a_8 + y.
    \end{align*} 
    Thus we've written each of $x$, $a_3$, and $z$ as sums of elements of $B'$.
    Note that $B$ has extended type $7 \hyph 5 \hyph 5 \hyph (4,4,3)$ but $B'$ has extended type $5 \hyph 5 \hyph 5 \hyph (2,3,3)$. 
\end{example}

\section[Quads and caps in AG(n,2)]{Quads and caps in $\AG(n,2)$}

The vector space $\Z_2^n$ is a model for the finite affine geometry $\AG(n,2)$.
We now recall the key notions of \deff{quad}, \deff{cap}, and associated terms from \cite{lamatpaper}.

\bigskip

\begin{definition} \leavevmode
    \begin{itemize} 
        \item A \deff{quad} in $\Z_2^n$ is a set of four elements that sum to zero, or equivalently an affine dependence relation among four elements.
        \item A \deff{cap} $C$ is a quad-free subset of $\Z_2^n$. 
        \item The \deff{first quad closure} of a subset $S \subseteq \Z_2^n$ is
        \[
            \qc_1(S) = S \cup \{x+y+z \mid \text{$x,y,z \in S$ are distinct}\}. 
        \]
        \item A cap $C$ is \deff{complete} if it is a maximal quad-free subset of $\aff(C)$, or equivalently if $ \qc_1(C) = \aff(C)$. 
        \item A cap $C$ is a \deff{$k$-cap} if $|C|=k$.
    \end{itemize}
\end{definition}

As we saw in the last section, a set $C$ can be partitioned, in a non-unique way, into a basis $B$ and a dependent set $D = C \setminus B$. Moreover, any element in $D$ can be written uniquely as the sum of an odd number of basis elements.  If $C$ is a cap, then no four elements can sum to zero, so no element of $C$ can be the sum of three basis elements. 

\bigskip

\begin{lemma} \label{morethan3}
    Let $C$ be a cap with basis $B$ and dependent set $D$. Then each element of $D$ can be written uniquely as a sum of an odd number of at least five elements of $B$.
\end{lemma}

\begin{proof}
    Let $C, B, D$ be as in the lemma statement.
    Let $x \in D$. Since $B$ is a basis for $\aff(C)$, we know $x$ can be written uniquely as an affine combination of elements of $B$. By Lemma~\ref{lem:odd-sum}, this combination is a sum of an odd number of elements of $B$. This odd number cannot be one, since by assumption $x \notin B$. This odd number cannot be three either, for if $x = a + b + c$ for distinct $a,b,c \in B$, then $a + b + c + x = \vec{0}$, so $\{a,b,c,x\}$ is a quad in the cap $C$, a contradiction. Thus, the odd number must be at least five.
\end{proof}

Our main goal in this paper is to classify caps of dimension $7$ up to affine equivalence. We start by recalling some results from \cite{lamatpaper}. 

Let $C$ be a $6$-dimensional cap. By changing coordinates, we can view $C$ as a subset of $\Z_2^6$, which corresponds to a standard quads deck. From Theorems~6.3 and~6.6 of \cite{lamatpaper}, we deduce the following results.

\bigskip

\begin{proposition} \label{capsinz6}
    If $C$ is a $6$-dimensional cap, then $7 \leq |C| \leq 9$. Moreover:
    \begin{enumerate} 
        \item There is one equivalence class each of $7$-caps and $9$-caps of dimension $6$.
        \item There are two equivalence classes of $8$-caps of dimension $6$.
    \end{enumerate}    
\end{proposition}

\bigskip

\begin{remark} \label{atmost2inD}
 Since a $6$-dimensional cap can contain only $7$, $8$, or $9$ elements, and a basis for it has size $7$, then $|D| \leq 2$.  
\end{remark}

\bigskip

\begin{proposition}[Theorem 6.1, \cite{lamatpaper}] \label{smallD-caps}
    Let $C$ be an $m$-dimensional $k$-cap in $\Z_2^n$.
    \begin{enumerate}
        \item If $k=m+1$, then  $C$ is affinely independent. All caps of this form comprise a single equivalence class.
        \item If $k=m$, then $C = B \cup {x}$, where $B$ is independent, and $x$ is the sum of an odd number $r$ of elements of $B$. There is an equivalence class of caps of this form for each odd number $5 \leq r \leq m+1$.
    \end{enumerate}
\end{proposition}

\section[Caps of dimension 7]{Caps of dimension $7$}

We now prove some general properties of $7$-dimensional caps, which will be used throughout the paper. 

For the following lemmas, let $C$ be a $7$-dimensional cap in $\Z_2^n$ with basis $B$ and dependent set $D = C \setminus B$. 
Note that by Lemma~\ref{morethan3}, we know that each element of $D$ can be written uniquely as the sum of exactly five or seven basis elements.

First, we show that two dependent elements cannot both be sums of seven basis elements.

\bigskip

\begin{lemma} \label{oneseven}
    At most one point in $D$ can be a sum of seven basis elements.
\end{lemma} 

\begin{proof} 
    Suppose $x, y \in D$ are both sums of seven basis elements. Then we can write $B = \{a_1, \ldots, a_8\}$ such that
    \begin{align*}
        x &= a_1 + \cdots + a_7, \\
        y &= a_2 + \cdots + a_8.
    \end{align*}
    Then $x + y = a_1 + a_8$, which implies that $C$ contains a quad, a contradiction. Hence $x$ and $y$ cannot both be sums of 7 basis elements.
\end{proof} 

The next lemma shows that two sums of five must share either two or three basis elements.

\bigskip

\begin{lemma} \label{twoorthree}
    Let $x_1, x_2 \in D$ be sums of five basis elements. Then $|B_{12}|$ equals $2$ or $3$.
\end{lemma} 

\begin{proof}
    By the inclusion-exclusion principle we have
    \[
        8
        = |B|
        \geq |B_1 \cup B_2|
        = |B_1| + |B_2| - |B_{12}|
        = 10 - |B_{12}|,
    \]
    so $|B_{12}| \geq 10 - 8 = 2$.

    Now suppose that $|B_{12}| = 4$.
    Then we can write $B = \{ a_1, \ldots, a_8\}$ such that
    \begin{align*}
        x_1 &= a_1 + \ldots + a_5, \\
        x_2 &= a_1 + \ldots + a_4 + a_6,
    \end{align*}
    so $B_{12} = \{a_1, \ldots, a_4\}$.
    Then $x_1 + x_2 = a_5 + a_6$ which implies that $C$ contains a quad, contradicting the fact that $C$ is a cap. Hence $|B_{12}| \neq 4$.
    
    Finally, since $x_1 \neq x_2$ we have $B_1 \neq B_2$, so $|B_{12}| \neq 5$.
    
    It follows that $|B_{12}|$ equals $2$ or $3$.
\end{proof}

The following lemma will be useful for constructing caps throughout the rest of the paper. 

\bigskip

\begin{lemma} \label{sumiseight}
    Let $x_1$, $x_2$, and $x_3$ be distinct elements of $D$. Then $|B_1 \cup B_2 \cup B_3| = 8$.
\end{lemma}

\begin{proof} 
    Suppose $|B_1 \cup B_2 \cup B_3| \leq 7$. Then the subcap $C'$ with basis $B' = B_1 \cup B_2 \cup B_3$ and dependent set $D' = \{x_1, x_2, x_3\}$ has dimension at most $6$. But by Proposition~\ref{capsinz6}, since $C'$ is a cap of dimension at most $6$, we have $|D'| \leq 2$. This is a contradiction since $|D'| = 3$. Hence it must be that $|B_1 \cup B_2 \cup B_3| = 8$.
\end{proof}

\bigskip

\begin{lemma} \label{sevenfivefour}
    Let $x_1, x_2 \in D$ with $|B_1| = 7$ and $|B_2|=5$. Then $|B_{12}| = 4$.
\end{lemma}

\begin{proof}
    By inclusion-exclusion, we have
    \[
        8
        = |B|
        \geq |B_1 \cup B_1|
        = |B_1| + |B_2| - |B_{12}|
        = 12 - |B_{12}|,
    \]
    so $|B_{12}| \geq 12 - 8 = 4$.

    Suppose $|B_{12}|=5$, the maximum possible.
    Then we can write $B = \{a_1, \ldots, a_8\}$ such that
    \begin{align*}
        x_1 &= a_1 + \ldots + a_7, \\
        x_2 &= a_1 + \ldots + a_5.
    \end{align*} 
    But then $x + y = a_6 + a_7$, which implies $C$ contains a quad, a contradiction. Thus $|B_{12}| = 4$. 
\end{proof}

\section[Caps of size 10]{Caps of size $10$}

We are now ready to begin our classification of $7$-dimensional caps into equivalence classes and types. Recall that a $7$-dimensional cap must contain a basis of size $8$, and hence must have cardinality at least $8$. Since that $8$-caps and $9$-caps in $\Z_2^7$ are characterized by Proposition~\ref{smallD-caps}, we assume that our caps have size $10$ or higher.

In this section, we prove that there are precisely two equivalence classes of $10$-caps in $\Z_2^7$. We do this by showing that any basis for such a cap must have one of three possible extended types, and constructing a dependent set template for each. These results are summarized in Table~\ref{dim7-10cap-templates}. (The constructions of the templates are in the proofs of Theorem~\ref{sevenisfive} and Corollary~\ref{size10twoclasses} below.)

\begin{table}[h]
    \centering
    \caption{Dependent set templates for $7$-dimensional $10$-caps with basis $\{a_1, \ldots, a_8\}$.}
    \label{dim7-10cap-templates}
    \begin{tabular}{|rl|}
        \hline
        \textbf{basis extended type} & \textbf{dependent set template} \\
        \hline
        $7 \hyph 5 \hyph (4)$
            & $x_1 = a_1 + \cdots + a_7$ \\
            & $x_2 = a_4 + \cdots + a_8$ \\
        \hline
        $5 \hyph 5 \hyph (2)$
            & $x_1 = a_1 + \cdots + a_5$ \\
            & $x_2 = a_4 + \cdots + a_8$ \\
        \hline
        $5 \hyph 5 \hyph (3)$
            & $x_1 = a_1 + \cdots + a_5$ \\
            & $x_2 = a_3 + \cdots + a_7$ \\
        \hline
    \end{tabular}
\end{table}


\bigskip

\begin{lemma}\label{no10in6}
    A $10$-cap in $\Z_2^n$ must have dimension at least $7$.
\end{lemma} 

\begin{proof} 
    This follows directly from Proposition~\ref{capsinz6}.
\end{proof} 

\bigskip

\begin{remark} \label{types_for_7dim10cap}
Let $C$ be a $7$-dimensional $10$-cap with basis $B$ and dependent set $D$. Then $|B|=8$ and $|D|=10-8=2$. 

By Lemmas~\ref{morethan3}, \ref{oneseven}, \ref{sevenfivefour}, and \ref{twoorthree}, the basis $B$ must have extended type $7 \hyph 5 \hyph (4)$, $5 \hyph 5 \hyph (2)$, or $5 \hyph 5 \hyph (3)$. 
\end{remark}

In the following theorem, we show that a cap with basis of type $7 \hyph 5$ then there exists a basis $B'$ of extended type $5 \hyph 5 \hyph (2)$.

\bigskip

\begin{theorem} \label{sevenisfive}
    Let $C$ be a $10$-cap of dimension $7$ with basis $B$ of extended type $7 \hyph 5 \hyph (4)$. Then there exists a basis $B'$ of extended type $5 \hyph 5 \hyph (2)$.
\end{theorem} 

\begin{proof} 
    Let $C$ and $B$ be as in the theorem statement, and let $D = \{x_1,x_2\}$. Then $|B_1| = 7$ and $|B_2| = 5$ and $|B_{12}| = 4$. So we can let $B = \{a_1, \ldots, a_8\}$ where
    \begin{align*} 
        x_1 &= a_1 + \cdots + a_7, \\
        x_2 &= a_4 + a_5 + \cdots + a_8.
    \end{align*} 
    Adding these together, we get 
    \[
        x_1 + x_2 = a_1 + a_2 + a_3 + a_8
    \]
    and so $x_1 = x_2 + a_1 + a_2 + a_3 + a_8$.
    In particular, this means $a_4 \in B_{12}$, so by the Affine Basis Exchange Theorem (\ref{newcapbasis}), we have that $B' = (B \setminus \{a_4\}) \cup \{x_2\}$ is also a basis for $C$,  with dependent set $D' = (D \setminus \{x_2\}) \cup \{a_4\} = \{a_4,x_1\}$. Note that
    \[
        a_4 = a_5 + \cdots + a_8 + x_2,
    \]
    so both $a_4$ and $x_1$ are sums of five basis elements in $B'$ with two elements in common.
    Hence $B'$ has extended type $5 \hyph 5 \hyph (2)$.
\end{proof} 

\bigskip

We can now characterize $10$-caps of dimension $7$.

\bigskip

\begin{theorem} \label{size10dim7}
    Let $C \subseteq \Z_2^n$ be a $7$-dimensional subset of size $10$. Then $C$ is a cap if and only if $C$ has a basis of extended type $5 \hyph 5 \hyph (2)$ or $5 \hyph 5 \hyph (3)$.
\end{theorem}

\begin{proof}
    Let $C$ be as in the theorem statement.

    Suppose $C$ is a cap. As described in Remark~\ref{types_for_7dim10cap}, we know $C$ must have a basis $B$ of extended type $7 \hyph 5 \hyph (4)$, $5 \hyph 5 \hyph (2)$, or $5 \hyph 5 \hyph (3)$. By Theorem~\ref{sevenisfive}, we may restrict these possibilities to just $5 \hyph 5 \hyph (2)$ and $5 \hyph 5 \hyph (3)$.

    \bigskip

    Now, suppose $C$ is any set as in the theorem statement.
    Let $B$ be a basis for $S$ with dependent set $D = \{x_1,x_2\}$, and suppose $B$ has type $5 \hyph 5 \hyph (2)$ or $5 \hyph 5 \hyph (3)$.
    To show that $C$ is a cap, we must show that $C$ does not contain a quad, i.e.~no four cap elements sum to zero.
    Equivalently, there are no affine dependence relations in $C$ involving exactly four elements. 


    Since the elements of $B$ are independent, the only affine dependence relations come from the expressions for $x_1$, $x_2$, and $x_1+x_2$ in as sums of basis elements.
    
    Suppose $B$ has extended type $5 \hyph 5 \hyph (2)$. Then we can set $B = \{a_1, \ldots, a_8\}$ such that
    \begin{equation} \label{eq:552}
    \begin{split}
        x_1 &= a_1 + \cdots + a_5, \\
        x_2 &= a_4 + a_5 + \cdots + a_8.
    \end{split}
    \end{equation}
    The expressions for $x_1$ and $x_2$ are each dependence relations with six elements, and
    \[
        x_1 + x_2 = a_1 + a_2 + a_3 + a_6 + a_7 + a_8
    \]
    gives a relation with eight elements. Therefore $C$ does not contain a quad, and must be a cap.

    Suppose $B$ has extended type $5 \hyph 5 \hyph (3)$. Then we can write $B = \{a_1, \ldots, a_8\}$ such that
    \begin{equation} \label{eq:553}
    \begin{split}
        x_1 &= a_1 + \cdots + a_5, \\
        x_2 &= a_3 + a_4 + a_5 + a_6 + a_7.
    \end{split}
    \end{equation}
    Again, the expressions for $x_1$ and $x_2$ are each relations with six elements, and
    \[
        x_1 + x_2 = a_1 + a_2 + a_6 + a_7
    \]
    is a relation with six elements. Therefore $C$ does not contain a quad, and must be a cap.
\end{proof}

\bigskip

\begin{corollary} \label{size10twoclasses}
    There are two equivalence classes of 7 dimensional $10$-caps. In one class, all caps have bases of both extended type $5 \hyph 5 \hyph (2)$ and $7 \hyph 5 \hyph (4)$.
    In the other class, all cap bases have extended type $5 \hyph 5 \hyph (3)$.
\end{corollary} 

\begin{proof}
    Let $C$ be a $10$-cap of dimension $7$. By Theorem~\ref{size10dim7}, $C$ has a basis $B$ of extended type either $5 \hyph 5 \hyph (2)$ or $5 \hyph 5 \hyph (3)$.

    If $B$ has extended type $5 \hyph 5 \hyph (3)$, by inclusion-exclusion we have
    \[
        |B_{1} \cup B_{2}| = |B_1| + |B_2| - |B_{12}| = 5 + 5 - 3 = 7,
    \]
    which means that one basis element does not appear in either $B_1$ or $B_2$, and hence it is not involved in any relations involving cap points. On the other hand, for caps with a basis $B$ of extended type $5 \hyph 5 \hyph (2)$, we have
    \[
        |B_{1} \cup B_{2}| = |B_1| + |B_2| - |B_{12}| = 5 + 5 - 2 = 8,
    \]
    which means that all elements of the cap occur in relations between other elements. By Remark~\ref{preserved-dependencies}, this property is preserved by affine equivalence. Hence, sets with bases of these two extended types cannot be affinely equivalent.

    On the other hand, in the proof of Theorem~\ref{size10dim7} we constructed dependent set templates for caps with a basis of extended type $5 \hyph 5 \hyph (2)$. Since there is a common template for all such caps, by Lemma~\ref{same-template} they must all be affinely equivalent. This proof also constructed a template for caps with bases of extended type $5 \hyph 5 \hyph (3)$, so these must all be affinely equivalent, as well.
\end{proof} 

Thus, we have exactly two equivalence classes of $10$-caps in $\Z_2^7$. Moreover, neither is complete, as we prove in Corollary~\ref{10capdim7incomplete} in the next section.


\section{Caps of size 11}

In this section, we prove that all $11$-caps in $\Z_2^7$ are equivalent. We do this by showing that any basis for such a cap must have one of three possible extended types, and constructing a dependent set template for each. These results are summarized in Table~\ref{dim7-11cap-templates}. (The constructions of the templates are in the proofs of Theorems~\ref{755is555} and~\ref{twotypes11caps} below.)

\begin{table}[h!]
    \centering
    \caption{Dependent set templates for $7$-dimensional $11$-caps with basis $\{a_1, \ldots, a_8\}$.} \label{dim7-11cap-templates}
    \begin{tabular}{|rl|}
        \hline
        \textbf{basis extended type} & \textbf{dependent set template} \\
        \hline
        $7 \hyph 5 \hyph 5 \hyph (4,4,3)$
            & $x_1 = a_1 + \cdots + a_7$ \\
            & $x_2 = a_4 + \cdots + a_8$ \\
            & $x_3 = a_1 + a_2 + a_6 + a_7 + a_8$ \\
        \hline
        $5 \hyph 5 \hyph 5 \hyph (3,3,3)$
            & $x_1 = a_1 + \cdots + a_5$ \\
            & $x_2 = a_3 + \cdots + a_7$ \\
            & $x_3 = a_1 + a_3 + a_4 + a_6 + a_8$ \\
        \hline
        $5 \hyph 5 \hyph 5 \hyph (3,3,2)$
            & $x_1 = a_1 + \cdots + a_5$ \\
            & $x_2 = a_3 + \cdots + a_7$ \\
            & $x_3 = a_1 + a_2 + a_3 + a_7 + a_8$ \\
        \hline
    \end{tabular}
\end{table}

Let $C$ be an $11$-cap of dimension $7$, with basis $B$ and dependent set $D$. Since $|B| = 8$, by Lemma~\ref{morethan3} each element of $D$ is a sum of five or seven basis elements, and by Lemma~\ref{oneseven}, at most one of is a sum of seven basis elements. Therefore $B$ must have type $5 \hyph 5 \hyph 5$ or $7 \hyph 5 \hyph 5$.

\bigskip

\begin{theorem} \label{755implies443}
    Suppose $C$ is an $11$-cap of dimension $7$ with basis $B$ of type $7 \hyph 5 \hyph 5$. Then it must have extended type $7 \hyph 5 \hyph 5 \hyph (4,4,3)$.
\end{theorem}

\begin{proof}
    Let $C$ and $B$ be as in the statement of the theorem. Then we can write $D = \{ x_1, x_2, x_3\}$ where $|B_1| = 7$, $|B_2| = 5$, and $|B_3| = 5$. We know $|B_{12}| = |B_{13}| = 4$ from Lemma~\ref{sevenfivefour} and that $|B_{23}| = 2$ or $|B_{23}| = 3$ from Lemma~\ref{twoorthree}.
    
    Suppose that $|B_{23}| = 2$. By Lemma~\ref{sumiseight}, we know $|B_1 \cup B_2 \cup B_3| = 8$. Then by inclusion-exclusion, we have
    \[
        |B_1 \cup B_2 \cup B_3| = |B_1| + |B_2| + |B_3| - (|B_{12}| + |B_{13}| + |B_{23}|) + |B_{123}|.
    \]
    Substituting the values from above, we get
    \[
        8 = (7 + 5 + 5) - (4 + 4 + 2) + |B_{123}|,
    \]
    so $|B_{123}| = 1$. 

    Then we can write $B = \{a_1, \ldots, a_7\}$ where
    \begin{align*} 
        x_1 & = a_1 + \ldots + a_7, \\ 
        x_2 &= a_4 + a_5 + a_6 + a_7 + a_8, \\
        x_3 &= a_1 + a_2 + a_3 + a_7 + a_8
    \end{align*}
    and hence $B_{123} = \{a_7\}$.
    Then $x_1 + x_2 + x_3 = a_7$, so $x_1, x_2, x_3, a_7$ form a quad, which is a contradiction. This shows $|B_{12}| \neq 2$, so $|B_{12}| = 3$.
\end{proof}

\bigskip

\begin{theorem} \label{755is555}
    Every $11$-cap of dimension $7$ has a basis of type $5 \hyph 5 \hyph 5$.
\end{theorem}  

\begin{proof}
    Let $C$ be an $11$-cap of dimension $7$, let $B$ be a basis for $C$, and let $D = C \setminus B$. As described above, we can write $D = \{x_1, x_2, x_3\}$ where $B$ has type $5 \hyph 5 \hyph 5$ or $7 \hyph 5 \hyph 5$. If the former, then the proof is complete. 
    
    Suppose $B$ has type $7 \hyph 5 \hyph 5$, so $x_1$ is a sum of seven basis elements and $x_2, x_3$ are each sums of five. Then by Lemma~\ref{sevenfivefour} and Theorem~\ref{755implies443}, we have $|B_{12}| = |B_{13}| = 4$ and $|B_{23}| = 3$, so we can write $B = \{a_1, \ldots, a_8\}$ such that
    \begin{align*}
        x_1 &= a_1 + \cdots + a_7, \\
        x_2 &= a_4 + a_5 + \cdots + a_8, \\
        x_3 &= a_1 + a_2 + a_6 + a_7 + a_8.
    \end{align*}
    Since $a_4 \in B_{12} \setminus B_{3}$, we can use the Affine Basis Exchange Theorem (\ref{newcapbasis}) to form the basis $B' = (B \setminus \{a_4\}) \cup \{x_2\}$ for $C$ with dependent set $D'= C \setminus B' = \{a_4, x_1, x_3\}$. The theorem also implies that
    \begin{align*}
        & B'_{1} = (B_{1} \symdif B_{2}) \cup \{x_2\} = \{x_2, a_1, a_2, a_3, a_8\}, \\
        & B'_{a_4} = (B_{2} \setminus \{a_4\}) \cup \{x_2\} = \{x_2, a_5, a_6, a_7, a_8\}, \\
        & B'_{3} = B_{3}.
    \end{align*}
    Then $|B'_{1}| = |B'_{a_4}| = |B'_{3}| = 5$, so each element of $D'$ is a sum of five basis elements in $B'$, and $B'$ has type $5 \hyph 5 \hyph 5$.
\end{proof} 

\bigskip

\begin{lemma} \label{tripintersect}
    Let $C$ be an $11$-cap of dimension $7$ with basis $B$ of type $5 \hyph 5 \hyph 5$.
    Then
    \[
        |B_{123}| = (|B_{12}| + |B_{13}| + |B_{23}|) - 7.
    \]
\end{lemma}

\begin{proof}
    Let $C$ and $B$ be as in the theorem statement, and let $D= \{x_1, x_2, x_3\}$ be the dependent set. 
    By inclusion-exclusion we have
    \begin{equation} \label{eq5}
        |B_1 \cup B_2 \cup B_3|
        = |B_1| + |B_2| + |B_3|
        - (|B_{12}| + |B_{13}| + |B_{23}|)
        + |B_{123}|.
    \end{equation}
    Since $B$ has type $5 \hyph 5 \hyph 5$, we know that
    $|B_1| + |B_2| + |B_3| = 5 + 5 + 5 = 15$.
    By Lemma~\ref{sumiseight}, we have that $|B_1 \cup B_2 \cup B_3| = 8$.
    Putting this together with equation~\eqref{eq5}, we have
    \[
        |B_{123}| = (|B_{12}| + |B_{13}| + |B_{23}|) - 7.
    \]
\end{proof}

Let $C$ be an $11$-cap of dimension $7$. By Lemmas~\ref{morethan3} and~\ref{oneseven}, any basis for $C$ must have type $7 \hyph 5 \hyph 5$ or $5 \hyph 5 \hyph 5$. By Theorem~\ref{755is555}, we can always find a basis $B$ of type $5 \hyph 5 \hyph 5$, and by Lemma~\ref{twoorthree} we know that each cardinality $|B_{ij}|$ is $2$ or $3$.
There are four possibilities for $C$:
\begin{description}
    \item[Case 1.] Each $B_{ij}$ has three elements. 
    \item[Case 2.] One $B_{ij}$ has two elements and the others each have three. 
    \item[Case 3.] One $B_{ij}$ has three elements and the others each have two.
    \item[Case 4.] Each $B_{ij}$ has two elements. 
\end{description}

It turns out that only the first two cases are valid, as we show in the following theorems.

\bigskip 

\begin{theorem}[Forbidden Triples] \label{forbiddentriples}
    Let $C$ be an $11$-cap of dimension $7$ with basis $B$ of type $5 \hyph 5 \hyph 5$. Then $|B_{ij}|= 2$ for at most one pair $\{i,j\}$.
\end{theorem}

\begin{proof}
    Let $C$ and $B$ be as in the theorem statement, and let $D= \{x_1, x_2, x_3\}$ be the dependent set. By Lemma~\ref{tripintersect}, we have
    \begin{equation} \label{tripleint}
        |B_{123}| = (|B_{12}| + |B_{13}| + |B_{23}|) - 7.
    \end{equation}
    We will show that the extended types besides those in the theorem statement are forbidden, i.e.~not valid.
    
    If $B$ has extended type $5 \hyph 5 \hyph 5 \hyph (2,2,2)$, then each $B_{ij}$ has two elements, so Equation~\eqref{tripleint} becomes $|B_{123}| = 6 - 7 = -1$, which is impossible.

    If $B$ has type $5 \hyph 5 \hyph 5 \hyph (3,2,2)$, then $|B_{12}|=3$ and $|B_{13}| = |B_{23}| = 2$. In this case, Equation~\eqref{tripleint} yields $|B_{123}| = 7 - 7 = 0$. Since $|B_{12}|=3$ we may write $B = \{a_1, \ldots, a_8\}$ such that
    \begin{align*}
        x_1 &= a_1 + \cdots + a_5, \\
        x_2 &= a_3 + \cdots + a_7.
    \end{align*}
    Since $|B_{123}| = 0$, it follows that $B_{123} = \emptyset$, so $a_3$, $a_4$, and $a_5$ are not in $B_3$. Since $|B_{13}| = |B_{23}| = 2$, and $B_3$ cannot contain $a_3$, $a_4$ or $a_5$, it must be that $a_1, a_2, a_6, a_7 \in B_3$. Since $|B_3|=5$, this means that 
    \[
        x_3 = a_1 + a_2 + a_6 + a_7 + a_8.
    \]
    But then $x_1 + x_2 + x_3 = a_8$,
    which implies $C$ contains a quad. This is a contradiction, so the extended type $5 \hyph 5 \hyph 5 \hyph (3,2,2)$ is also impossible. By similar arguments, the extended types $5 \hyph 5 \hyph 5 \hyph (2,3,2)$ and $5 \hyph 5 \hyph 5 \hyph (2,2,3)$ are impossible, as well.
\end{proof}

We can now show that there are at most two equivalence classes of $11$-caps of dimension $7$. 

\bigskip

\begin{theorem} \label{twotypes11caps}
    Let $C \subseteq \Z_2^n$ be a subset of size $11$ and dimension $7$. Then $C$ is a cap if and only if $C$ has a basis of extended type $5 \hyph 5 \hyph 5 \hyph (3,3,3)$ or $5 \hyph 5 \hyph 5 \hyph (3,3,2)$.
\end{theorem}

Note that if a subset $C$ has a basis of extended type $5 \hyph 5 \hyph 5 \hyph (2,3,3)$ or $5 \hyph 5 \hyph 5 \hyph (3,2,3)$, then we can always reorder the dependent elements so that the extended type is $5 \hyph 5 \hyph 5 \hyph (3,3,2)$. This is why we do not need to address these other extended types in the theorem above.

\begin{proof}
    Let $C$ be as in the theorem statement. One direction of the theorem follows immediately from Theorem~\ref{forbiddentriples}. For the other direction, suppose that $C$ has a basis of extended type $5 \hyph 5 \hyph 5 \hyph (3,3,3)$ or $5 \hyph 5 \hyph 5 \hyph (3,3,2)$, where $D = C \setminus B = \{x_1, x_2, x_3\}$.

    \bigskip

    Suppose $B$ is a basis for $C$ of extended type $5 \hyph 5 \hyph 5 \hyph (3,3,3)$. Then $|B_{ij}| = 3$ for all pairs $\{i,j\}$, and by Lemma~\ref{tripintersect} we know $|B_{123}| = (3+3+3) - 7 = 2$. Hence we can write $B = \{a_1, \ldots, a_8\}$ such that
    \begin{align*} 
        x_1 &= a_1 + \cdots + a_5, \\
        x_2 &= a_3 + \cdots + a_7.
    \end{align*} 
    Since $|B_{123}| = 2$, we know $B_3$ contains two of $a_3$, $a_4$, and $a_5$, so without loss of generality, suppose it contains $a_3, a_4 \in B_3$. The set $B_3$ additionally contains another element of $B_1$ and $B_2$, respectively, so without loss of generality we may write
    \[
        x_3 = a_1 + a_3 + a_4 + a_6 + a_8.
    \]

    Now we show that $C$ is a cap by checking that there are no affine dependencies among the elements of $C$ involving precisely four elements. All affine dependencies in $C$ are the result of linear combinations of the dependent elements, which are 
    \begin{equation} \label{dep-combos}
        \vec{0}, \, x_1, \, x_2, \, x_3, \, x_1+x_2, \, x_1+x_3, \, x_2+x_3, \, x_1+x_2+x_3.
    \end{equation}
    Observe that each set $B \cup \{x_i, x_j\}$ has size $10$ and dimension $7$ with basis $B$ of extended type $5 \hyph 5 \hyph (3)$, which by Theorem~\ref{size10dim7} must be a cap. It follows that every affine dependence without all of $x_1, x_2, x_3$ must contain more or less than four elements of $C$.
    Since
    \[
        x_1 + x_2 + x_3 = a_2 + a_3 + a_4 + a_7 + a_8
    \]
    involves eight elements, it follows that $C$ does not contain a quad, hence it is a cap.

    \bigskip

    Suppose $B$ is a basis for $C$ of extended type $5 \hyph 5 \hyph 5 \hyph (3,3,2)$. Then $|B_{12}| = |B_{13}| = 3$ and $|B_{23}| = 2$, and by Lemma~\ref{tripintersect} we know $|B_{123}| = (3+3+2) - 7 = 1$. Hence we can write $B = \{a_1, \ldots, a_8\}$ such that
    \begin{align*}
        x_1 &= a_1 + a_2 + a_3 + a_4 + a_5, \\
        x_2 &= a_3 + a_4 + a_5 + a_6 + a_7.
    \end{align*} 
    Since $|B_{123}| = 1$, we can assume without loss of generality that $B_3$ contains $a_3$ but not $a_4$ or $a_5$. Since $|B_{13}| = 3$ and $|B_{23}| = 2$, it must be that $B_3$ contains $a_1$ and $a_2$, and exactly one of $a_6$ and $a_7$. So without loss of generality, we can write
    \[
        x_3 = a_1 + a_2 + a_3 + a_7 + a_8.
    \]

    Now we show that $C$ is a cap. As above, we must check that none of the elements in \eqref{dep-combos} result in affine dependencies with exactly four elements of $C$.
    Observe that each set $B \cup \{x_i, x_j\}$ has size $10$ and dimension $7$ with basis $B$ of extended type $5 \hyph 5 \hyph (2)$ or $5 \hyph 5 \hyph (3)$, which by Theorem~\ref{size10dim7} must be a cap. It follows that every affine dependence without all of $x_1, x_2, x_3$ must contain more or less than four elements of $C$.
    Since
    \[
        x_1 + x_2 + x_3 = a_3 + a_6 + a_8
    \]
    involves six elements, it follows that $C$ does not contain a quad, hence it is a cap.
\end{proof}

The following corollary, which follows directly from the previous proof, will be used throughout the rest of the paper.

\bigskip

\begin{corollary} \label{triples}
    Let $C$ be a cap of dimension $7$ with basis $B$ of type $5 \hyph 5 \hyph 5$.
    \begin{enumerate}
        \item If $|B_{ij}| = 2$ for some pair $\{i,j\}$, then $|B_{123}| = 1$.
        \item If $|B_{ij}| = 3$ for all pairs $\{i,j\}$, then $|B_{123}| = 2$.
    \end{enumerate} 
\end{corollary}

\bigskip

We can also now show that no $10$-cap can be complete.

\bigskip

\begin{corollary} \label{10capdim7incomplete}
     No $10$-cap is complete. 
\end{corollary} 

\begin{proof}
    Let $C$ be a $10$-cap. By Lemma~\ref{no10in6}, we know $\dim(C) \geq 7$. Suppose $\dim(C) = 7$. Then by Theorem~\ref{size10dim7} and its proof, we know $C$ has a basis $B = \{a_1, \ldots, a_8\}$ with dependent set $D = \{x_1, x_2\}$ where either
    \begin{equation} \label{552eqns}
    \begin{split}
        x_1 &= a_1 + \cdots + a_5, \\
        x_2 &= a_4 + \cdots + a_8
    \end{split}
    \end{equation}
    (if the basis has extended type $5 \hyph 5 \hyph (2)$) or
    \begin{equation} \label{553eqns}
    \begin{split}
        x_1 &= a_1 + \cdots + a_5, \\
        x_2 &= a_3 + \cdots + a_7
    \end{split}
    \end{equation}
    (if the basis has extended type $5 \hyph 5 \hyph (3)$).
    
    Suppose $x_1,x_2$ satisfy the equations in \eqref{552eqns}. Let $x_3 = a_1 + a_2 + a_5 + a_6 + a_7$. Since this is an affine combination of elements of $C$, we know $x_3 \in \aff(C)$. Then $C' = C \cup \{x_3\}$ is a subset of $\aff(C)$ with basis $B$ of extended type $5 \hyph 5 \hyph 5 \hyph 5 \hyph (2,3,3)$, or using the ordering $x_1,x_3,x_2$, extended type 
    $5 \hyph 5 \hyph 5 \hyph 5 \hyph (3,3,2)$.
    By Theorem~\ref{twotypes11caps} it follows that $C'$ is an $11$-cap.

    Suppose $x_1,x_2$ satisfy the equations in \eqref{553eqns}. Let $x_3 = a_1 + a_3 + a_6 + a_7 + a_8$. Since this is an affine combination of elements of $C$, we know $x_3 \in \Z_2^7$. Then $C' = C \cup \{x_3\}$ is a subset of $\Z_2^7$ with basis $B$ of extended type $5 \hyph 5 \hyph 5 \hyph 5 \hyph (3,3,2)$, so by Theorem~\ref{twotypes11caps} it follows that $C'$ is an $11$-cap.

    Therefore, in either case, the cap $C$ can be extended to a larger cap of the same dimension, so it cannot be complete.

    \bigskip

    Now suppose that $\dim(C) > 7$. Note that the first quad closure of $C$ has at most
    \[
        |C| + \binom{|C|}{3}
        = 10 + \binom{10}{3}
        = 130
    \]
    elements. But $\aff(C)$ has $2^{\dim(C)} \geq 2^8 = 256$ elements, so it is impossible for the first quad closure of $C$ to equal $\aff(C)$.
\end{proof}

Now we show that, in contrast to Corollary~\ref{size10twoclasses}, the two extended types for bases of $11$-caps described in Theorem~\ref{twotypes11caps} constitute a single equivalence class.

\bigskip 

\begin{proposition} \label{just555332}
    Every $11$-cap of dimension $7$ has a basis of extended type $5 \hyph 5 \hyph 5 \hyph (3,3,2)$.
\end{proposition}


\begin{proof}
    Let $C$ be an $11$-cap of dimension $7$. By Theorem~\ref{twotypes11caps}, we know $C$ has a basis of extended type $5 \hyph 5 \hyph 5 \hyph (3,3,2)$ or $5 \hyph 5 \hyph 5 \hyph (3,3,3)$. If the former, then the proof is complete.

    Suppose $C$ has a basis $B$ of extended type $5 \hyph 5 \hyph 5 \hyph (3,3,3)$. 
    From the proof of Theorem~\ref{twotypes11caps}, we can write $B = \{a_1, \ldots, a_8\}$ such that
    \begin{align*}
        x_1 &= a_1 + a_2 + a_3 + a_4 + a_5, \\ 
        x_2 &= a_1 + a_2 + a_3 + a_6 + a_7, \\ 
        x_3 &= a_1 + a_2 + a_4 + a_6 + a_8,
    \end{align*}
    where $x_1,x_2,x_3$ are the dependent elements.
    From this, we obtain
    \begin{align*} 
        x_1 + x_2 &= a_4 + a_5 + a_6 + a_7, \\
        x_1 + x_3 &= a_3 + a_5 + a_6 + a_8.
    \end{align*} 

    Let $B' =\{a_2, a_3, a_5, a_6, a_7, x_1, x_2, a_8\}$, and $D' = C \setminus B' = \{a_1, a_4, x_3\}$.
    Rearranging the equations above, we can write the elements of $D'$ as
    \begin{align*}
        a_4 &= a_5 + a_6 + a_7 + x_1 + x_2, \\ 
        a_1 &= a_2 + a_3 + a_6 + a_7 + x_2, \\
        x_3 &= a_3 + a_5 + a_6 + x_1 + a_8. 
    \end{align*}
    Since $|B'| = 8$ and the elements of $D'$ are affine combinations of elements of $B'$, the set $B'$ is a basis for $C$.
    Observe that
    \begin{align*} 
        |B_{a_4} \cap B_{a_1}| &= 3, \\
        |B_{a_4} \cap B_{x_3}| &= 3, \\ 
        |B_{a_1} \cap B_{x_3}| &= 2. 
    \end{align*} 
    Thus $C$ has a basis $B'$ of extended type $5 \hyph 5 \hyph 5 \hyph (3,3,2)$.
\end{proof}

\bigskip

\begin{corollary} \label{one11cap}
    All $11$-caps of dimension $7$ are equivalent.
\end{corollary}

\begin{proof}
    By Proposition~\ref{just555332}, we know every $11$-cap of dimension $7$ has a basis of extended type $5 \hyph 5 \hyph 5 \hyph (3,3,2)$. In the proof of Theorem~\ref{twotypes11caps}, we constructed a dependent set template for all such caps and bases. By Lemma~\ref{same-template}, this means all such caps are affinely equivalent. Hence there is only one equivalence class of $11$-caps in $\Z_2^7$.
\end{proof}

\bigskip


\section{Caps of size 12}

In this section, we prove that all $12$-caps in $\Z_2^7$ are equivalent. We do this by showing that any basis for such a cap must have one of three possible extended types, and constructing a dependent set template for each. These results are summarized in Table~\ref{dim7-12cap-templates}. (The constructions of the templates are in the proofs of Theorems~\ref{7555implies5555} and~\ref{two12caps-v2} below.)

\begin{table}[h]
    \centering
    \caption{Dependent set templates for $7$-dimensional $12$-caps with basis $\{a_1, \ldots, a_8\}$.} \label{dim7-12cap-templates}
    \begin{tabular}{|rl|}
        \hline
        \textbf{basis extended type} & \textbf{dependent set template} \\
        \hline
        $7 \hyph 5 \hyph 5 \hyph 5 \hyph (4,4,4,3,3,3)$
            & $x_1 = a_1 + \cdots + a_7$ \\
            & $x_2 = a_4 + \cdots + a_8$ \\
            & $x_3 = a_1 + a_2 + a_6 + a_7 + a_8$ \\
            & $x_4 = a_1 + a_3 + a_5 + a_7 + a_8$ \\
        \hline
        $5 \hyph 5 \hyph 5 \hyph 5 \hyph (2,3,3,3,3,3)$
            & $x_1 = a_1 + \cdots + a_5$ \\
            & $x_2 = a_1 + a_2 + a_3 + a_7 + a_8$ \\
            & $x_3 = a_3 + \cdots + a_7$ \\
            & $x_4 = a_2 + a_3 + a_4 + a_6 + a_8$ \\
        \hline
        $5 \hyph 5 \hyph 5 \hyph 5 \hyph (2,3,3,3,3,2)$
            & $x_1 = a_1 + \cdots + a_5$ \\
            & $x_2 = a_1 + a_2 + a_3 + a_7 + a_8$ \\
            & $x_3 = a_3 + \cdots + a_7$ \\
            & $x_4 = a_2 + a_4 + a_6 + a_7 + a_8$ \\
        \hline
    \end{tabular}
\end{table}


\bigskip 

\begin{theorem} \label{7555implies5555}
    Every $12$-cap of dimension $7$ has a basis of type $5 \hyph 5 \hyph 5 \hyph 5$.
\end{theorem} 

\begin{proof} 
    Let $C$ be an $11$-cap of dimension $7$, with basis $B$ and dependent set $D$. Since $|B| = 8$, by Lemma~\ref{morethan3} each element of $D$ is a sum of five or seven basis elements, and by Lemma~\ref{oneseven}, at most one of is a sum of seven basis elements. Therefore $B$ must have type $5 \hyph 5 \hyph 5 \hyph 5$ or $7 \hyph 5 \hyph 5 \hyph 5$. If the former, then the proof is complete. 
    
    Suppose $B$ has type $7 \hyph 5 \hyph 5 \hyph 5$. Let $D = \{x_1, x_2, x_3, x_4\}$. By Lemma~\ref{sevenfivefour} we have $|B_{12}| = 4$, so we can write $B = \{a_1, \ldots, a_8\}$ with
    \begin{align*}
        x_1 &= a_1 + \cdots + a_7, \\
        x_2 &= a_4 + \cdots + a_8.
    \end{align*}
    By the same lemma we know $|B_{13}| = |B_{14}| = 4$. Applying Theorem~\ref{755implies443} to the $7$-dimensional $10$-cap $C \setminus \{x_1\}$, we have that $|B_{23}| = |B_{24}| = |B_{34}| = 3$. Again, from the proof of Theorem~\ref{755is555} we have without loss of generality that
    \[
        x_3 = a_1 + a_2 + a_6 + a_7 + a_8
    \]
    and by a similar argument that, without loss of generality, 
    \[
        x_4 = a_1 + a_3 + a_5 + a_7 + a_8.
    \]
    Since $a_4 \in B_1 \cap B_2$ but $a_4 \notin B_3 \cup B_4$, we can apply the Affine Basis Exchange Theorem (\ref{newcapbasis}) to obtain the new basis
    \[
        B' = (B \setminus \{a_4\}) \cup \{x_2\},
    \]
    for $C$ with dependent set $D' = C \setminus B' = \{a_4, x_1, x_3, x_4\}$ and 
    \begin{align*}
       & B'_{1} = (B_{1} \symdif B_{2}) \cup \{x_2\} = \{x_2, a_1, a_2, a_3, a_8\}, \\
       & B'_{x_2} = (B_2 \setminus \{a_4\}) \cup \{x_2\} = \{x_2, a_5, a_6, a_7, a_8\}, \\
       & B'_{3} = B_{3}, \\
       & B'_{4} = B_{4}.
    \end{align*}
    Then $|B'_{1}| = |B'_{x_2}| = |B'_{3}| = |B'_{4}| = 5$, so $B'$ has type $5 \hyph 5 \hyph 5 \hyph 5$.
\end{proof}

By Lemma~\ref{twoorthree}, we know that for a $12$-cap with basis $B$ of type $5 \hyph 5 \hyph 5 \hyph 5$, each $|B_{ij}|$ equals $2$ or $3$. We now show that $|B_{ij}| = 2$ for at least one pair $\{i,j\}$.

\bigskip 

\begin{lemma}[Forbidden Quadruples] \label{forbiddenquadruples} 
    A $12$-cap of dimension $7$ cannot have a basis of extended type $5 \hyph 5 \hyph 5 \hyph 5 \hyph (3,3,3,3,3,3)$.
\end{lemma}

\begin{proof}
    Let $C$ be a $12$-cap of dimension $7$ with basis $B$ of type $5 \hyph 5 \hyph 5 \hyph 5$ and dependent set $D = \{x_1, x_2, x_3, x_4\}$. Suppose for the sake of contradiction that $|B_{ij}| = 3$ for each pair $\{i,j\}$. Then by the proof of Theorem~\ref{twotypes11caps} applied to the $7$-dimensional $11$-cap $C \setminus \{x_4\}$, we can write $B = \{a_1, \ldots, a_8\}$ such that
    \begin{align*}
        x_1 &= a_1 + a_2 + a_3 + a_4 + a_5, \\ 
        x_2 &= a_1 + a_2 + a_3 + a_6 + a_7, \\ 
        x_3 &= a_1 + a_2 + a_4 + a_6 + a_8.
    \end{align*} 
    Now we show that with $x_1, x_2, x_3$ given as above, the element $x_4$ is completely determined as a sum of five basis elements. 
    
    First, by Lemma~\ref{sumiseight}, we know that $|B_i \cup B_j \cup B_k| = 8$ for each triple $\{i,j,k\}$. Since $a_8 \notin B_1 \cup B_2$, it follows that $a_8 \in B_4$. Similarly, $a_5, a_7 \in B_4$.

    By Corollary~\ref{triples}, we know that $|B_{ijk}| = 2$ for each triple $\{i,j,k\}$.
    Using the inclusion-exclusion principle, we compute 
    \begin{align*}
        |B_{1234}|
        &= \sum |B_i| - \sum |B_{ij}| + \sum |B_{ijk}|
        - |B_1 \cup B_2 \cup B_3 \cup B_4| \\
        &= 20 - 3 \cdot \binom{4}{2} + 2 \cdot \binom{4}{3} - 8 \\
        &= 2.
    \end{align*}
    Since $B_{1234} \subseteq B_{123}$ and $|B_{123}| = |B_{1234}| = 2$, we have
    \[
        B_{1234} = B_{123} = \{a_1, a_2\}.
    \]
    Then $B_4 = \{a_1, a_2, a_5, a_7, a_8\}$ so
    \[
        x_4 = a_1 + a_2 + a_5 + a_7 + a_8.
    \]
    But then $x_1 + x_2 + x_3 + x_4 = 0$, which implies that $C$ contains a quad, which contradicts the fact that $C$ is a cap.
\end{proof}

\bigskip

\begin{proposition} \label{two12caps}
Let $C$ be a $12$-cap of dimension $7$. Then $C$ has a basis $B$ of type $5 \hyph 5 \hyph 5 \hyph 5$, and for any such basis the numbers $|B_{ij}|$ must satisfy the following.
\begin{enumerate}
    \item Exactly one or two of these numbers are $2$ and the rest are $3$.
    \item If $|B_{ij}| = |B_{k\ell}| = 2$, then $\{i,j\}$ and $\{k,\ell\}$ are either equal or disjoint.
\end{enumerate}
\end{proposition}

\begin{proof}
    Let $C$ be a $12$-cap. By Theorem~\ref{7555implies5555}, we know that $C$ has a basis $B$ of type $5 \hyph 5 \hyph 5 \hyph 5$. Let $D = \{x_1, x_2, x_3, x_4\}$. By Lemma~\ref{twoorthree} we know that $|B_{ij}|$ equals $2$ or $3$ for all pairs $\{i,j\}$, and by Lemma~\ref{forbiddenquadruples} they cannot all be $3$.
    
    For any triple $\{i,j,k\}$, the subset $C' = B \cup \{x_i, x_j, x_k\}$ of $C$ is a $7$-dimensional $11$-cap with basis $B$ of type $5 \hyph 5 \hyph 5$. Then by Theorem~\ref{forbiddentriples}, at most one of $|B_{ij}|$, $|B_{ik}|$, and $|B_{jk}|$ equals $2$. It follows that of the six numbers $|B_{ij}|$, where $\{i,j\} \subseteq \{1,2,3,4\}$, only one or two of them can be $2$ and the rest must be $3$. Furthermore, if two of these numbers are $2$, then the two pairs $\{i,j\}$ and $\{k,\ell\}$ for which $|B_{ij}| = |B_{k\ell}| = 2$ must be disjoint pairs, for the following reason. If two of $i,j,k,\ell$ are equal, say $j = k$, then in the triple $\{i,j,\ell\}$ we have two pairs $\{i,j\}$ and $\{j,\ell\}$ with $|B_{ij}| = |B_{j\ell}| = 2$. This is impossible, as described above.
    
    Therefore either $|B_{ij}| = 2$ for exactly one pair $\{i,j\} \subseteq \{1,2,3,4\}$, or else $|B_{ij}| = |B_{k\ell}| = 2$ for exactly two pairs $\{i,j\}$ and $\{k,\ell\}$, and these pairs are disjoint.
\end{proof}
    
\bigskip

\begin{theorem} \label{two12caps-v2}
    Let $C \subseteq \Z_2^n$ be a subset of size $12$ and dimension $7$. Then $C$ is a cap if and only if $C$ has a basis of extended type $5 \hyph 5 \hyph 5 \hyph 5 \hyph (2,3,3,3,3,3)$ or $5 \hyph 5 \hyph 5 \hyph 5 \hyph (2,3,3,3,3,2)$, up to a permutation of dependent elements.
\end{theorem}

\bigskip

Note that the extended types mentioned in the theorem above are particular cases of the two possibilities described in Proposition~\ref{two12caps}. For the first extended type, we have $B_{ij} = B_{k\ell} = B_{12}$, and for the second we have $B_{ij} = B_{12}$ and $B_{k\ell} = B_{34}$.

\bigskip

\begin{proof}
    Let $C$ be as in the theorem statement. One direction of the theorem follows immediately from Theorem~\ref{two12caps}, by reordering the dependent elements if necessary.
    
For the other direction, suppose that $C$ has a basis $B$ of extended type $5 \hyph 5 \hyph 5 \hyph 5 \hyph (2,3,3,3,3,3)$ or $5 \hyph 5 \hyph 5 \hyph 5 \hyph (2,3,3,3,3,3,2)$. Let $D = \{x_1, x_2, x_3, x_4\}$.
    For any triple $\{i,j,k\}$, the set $C_{ijk} = B \cup \{x_i, x_j, x_k\}$ has a basis $B$ of extended type $5 \hyph 5 \hyph 5 \hyph (3,3,3)$ or $5 \hyph 5 \hyph 5 \hyph (3,3,2)$, up to reordering the dependent elements, so by Theorem~\ref{twotypes11caps} we know $C_{ijk}$ is a cap. Then by Lemma~\ref{tripintersect}, we also have
    \begin{equation} \label{eq:ijk}
        |B_{ijk}| = |B_{ij}| + |B_{ik}| + |B_{jk}| - 7.
    \end{equation}
    Since the set $C_{123} = B \cup \{x_1, x_3, x_2\}$ is a $7$-dimensional $11$-cap with basis $B$ of type $5 \hyph 5 \hyph 5 \hyph (3,3,2)$ --- (note the ordering of dependent elements) --- by Table~\ref{dim7-11cap-templates} we can write $B = \{a_1,\ldots,a_8\}$ such that
    \begin{align*}
        x_1 &= a_1 + \cdots + a_5, \\
        x_3 &= a_3 + \cdots + a_7, \\
        x_2 &= a_1 + a_2 + a_3 + a_7 + a_8.
    \end{align*}
    Finally, note that Lemma~\ref{sumiseight} implies that $|B_1 \cup B_2 \cup B_3 \cup B_4| = 8$.

    We now consider the two extended type possibilities separately.

    \begin{description}
    \item[Case 1.]
    Suppose that $B$ is a basis for $C$ of extended type $5 \hyph 5 \hyph 5 \hyph 5 \hyph (2,3,3,3,3,3)$.
    By equation~\ref{eq:ijk}, we have $|B_{123}| = |B_{124}| = (2+3+3) - 7 = 1$ and $|B_{ijk}| = (3+3+3) - 7 = 2$ for all other triples $\{i,j,k\}$.
    To determine a formula for $x_4$, we first compute $|B_{1234}|$ using inclusion-exclusion. We know $|B_{12}|=2$ and $|B_{ij}|=3$ for all other pairs $\{i,j\}$, and that $|B_{ijk}| = 2$ for all triples $\{i,j,k\}$ except that $|B_{123}| = |B_{234}| = 1$. So
    \begin{align*}
        |B_{1234}|
        &= \sum |B_i| - \sum |B_{ij}| + \sum |B_{ijk}| - |B_1 \cup B_2 \cup B_3 \cup B_4| \\
        &= 4(5) - (3+3+2+3+3+3) + (1+1+2+2) - 8 \\
        &= 1.
    \end{align*}
    Now we can construct a formula for $x_4$ in the following steps.
    \begin{itemize}
        \item Since $B_{1234} \subseteq B_{123}$ and $|B_{1234}| = |B_{123}| = 1$, it follows that $B_{1234} = B_{123} = \{a_3\}$. So $B_4$ must contain $a_3$.
        
        \item Since $B_{124} \subseteq B_{12} = \{a_1, a_2, a_3\}$ and $|B_{124}| = 2$ and $a_3 \in B_4$, it follows that $a_1$ or $a_2$ is in $B_4$, but not both. Without loss of generality, assume $a_2 \in B_4$ and $a_1 \notin B_4$.
        
        \item Since $B_{134} \subseteq B_{13} = \{a_3, a_4, a_5\}$ and $|B_{134}| = 2$ and $a_3 \in B_4$, either $a_4$ or $a_5$ in in $B_4$, but not both. Without loss of generality, assume $a_4 \in B_4$ and $a_5 \notin B_4$
        
        \item Since $B_{234} \subseteq B_{23} = \{a_3, a_7\}$ and $|B_{234}| = 1$ and $a_3 \in B_4$, it follows that $a_7 \notin B_4$.
        
        \item By the previous bullet points, we have $a_2, a_3, a_4 \in B_4$ and $a_1, a_5, a_7 \notin B_4$. Since $|B_4| = 5$, it follows that $a_6, a_8 \in B_4$.

    \end{itemize}
    Therefore we can write
    \[
        x_4 = a_2 + a_3 + a_4 + a_6 + a_8.
    \]

    Now we show that $C$ is a cap by checking that there are no affine dependence relations among the elements of $C$ involving precisely four elements. All affine dependence relations in $C$ are the result of linear combinations of the dependent elements. But since each set $C_{ijk} = B \cup \{x_i, x_j, x_k\}$ is a cap, we already know that the affine dependence relations generated by just three dependent elements at a time cannot involve precisely four elements. The only one left to check is
    \[
        x_1 + x_2 + x_3 + x_4
        = a_2 + a_4,
    \]
    which involved six elements. It follows that $C$ does not contain a quad, hence it is a cap.

    \item[Case 2.]
    (This argument is similar to that for Case 1, with a few crucial differences.) Suppose that $B$ is a basis for $C$ of extended type $5 \hyph 5 \hyph 5 \hyph 5 \hyph (2,3,3,3,3,2)$. Then $|B_{12}| = |B_{34}| = 2$, so  equation~\eqref{eq:ijk} implies that $|B_{ijk}| = (2+3+3) - 7 = 1$ for all triples $\{i,j,k\}$.
    To determine a formula for $x_4$, we first compute $|B_{1234}|$ using inclusion-exclusion. We know $|B_{12}|=|B_{34}|=2$ and $|B_{ij}|=3$ for all other pairs $\{i,j\}$, and that $|B_{ijk}| = 2$ for all triples $\{i,j,k\}$. So
    \begin{align*}
        |B_{1234}|
        &= \sum |B_i| - \sum |B_{ij}| + \sum |B_{ijk}| - |B_1 \cup B_2 \cup B_3 \cup B_4| \\
        &= 4(5) - (2+3+3+3+3+2) + (1+1+1+1) - 8 \\
        &= 0.
    \end{align*}
    Now we can construct a formula for $x_4$ in the following steps.
    \begin{itemize}
        \item Since $B_{1234} \subseteq B_{123} = \{a_3\}$ and $|B_{1234}| = 0$, it follows that $a_3 \notin B_4$.
        \item Since $B_{124} \subseteq B_{12} = \{a_1, a_2, a_3\}$ and $|B_{124}| = 1$ and $a_3 \notin B_4$, it follows that $a_1$ or $a_2$ is in $B_4$, but not both. Without loss of generality, assume $a_2 \in B_4$ and $a_1 \notin B_4$.
        
        \item Since $B_{134} \subseteq B_{13} = \{a_3, a_4, a_5\}$ and $|B_{134}| = 1$ and $a_3 \notin B_4$, it follows that $a_4$ or $a_5$ is in $B_4$, but not both. Without loss of generality, assume $a_4 \in B_4$ and $a_5 \notin B_4$
        
        \item Since $B_{234} \subseteq B_{23} = \{a_3, a_7\}$ and $|B_{234}| = 1$ and $a_3 \notin B_4$. it follows that $a_7 \in B_4$.
        
        \item By the previous bullet points, we have $a_2, a_4, a_7 \in B_4$ and $a_1, a_3, a_5 \notin B_4$. Since $|B_4| = 5$, it follows that $a_6, a_8 \in B_4$.

    \end{itemize}
    Therefore we can write
    \[
        x_4 = a_2 + a_4 + a_6 + a_7 + a_8.
    \]

    Now we show that $C$ is a cap by checking that there are no affine dependence relations among the elements of $C$ involving precisely four elements. All affine dependence relations in $C$ are the result of linear combinations of the dependent elements. But since each set $C_{ijk} = B \cup \{x_i, x_j, x_k\}$ is a cap, we already know that the affine dependence relations generated by just three dependent elements at a time cannot involve precisely four elements. The only one left to check is
    \[
        x_1 + x_2 + x_3 + x_4
        = a_2 + a_3 + a_4 + a_7,
    \]
    which involves eight elements. It follows that $C$ does not contain a quad, hence it is a cap.
    \end{description}
\end{proof}

The next corollary gives us the cardinality of $B_{1234}$ for any four sums of five in a cap.

\bigskip

\begin{corollary} \label{quads}
    Let $C$ be a cap of dimension $7$ with basis $B$.  Let $x_1, x_2, x_3, x_4 \in D$ each be sums of five basis elements. Then
    \begin{enumerate}
        \item If $|B_{ij}| = 2$ for exactly one pair $\{i,j\}$  then $|B_{1234}| = 1$.
        \item If $|B_{ij}| = 2$ for two disjoint pairs $\{i,j\}$  then $|B_{1234}| = 0$.
    \end{enumerate} 
\end{corollary}

\bigskip

We can now show that an $11$-cap cannot be complete.

\bigskip 

\begin{corollary} \label{11notcomplete}
    No $11$-cap is complete.
\end{corollary} 

\begin{proof}
    Let $C$ be an $11$-cap. By Lemma~\ref{no10in6}, we know $\dim(C) \geq 7$. Suppose $\dim(C) = 7$. Then by Proposition~\ref{just555332} we know $C$ has a basis $B$ of extended type $5 \hyph 5 \hyph 5 \hyph (3,3,2)$. By Table~\ref{dim7-11cap-templates} we can write $B = \{a_1, \ldots, a_8\}$ and $D = \{x_1, x_2, x_3\}$ where
    \begin{align*}
        x_1 &= a_1 + \cdots + a_5, \\
        x_2 &= a_3 + \cdots + a_7, \\
        x_3 &= a_1 + a_2 + a_3 + a_6 + a_8.
    \end{align*}
    Let $x_4 = a_2 + a_3 + a_4 + a_6 + a_8$. Since this is an affine combination of elements of $C$, we know that $x_4 \in \aff(C)$. Then $C' = C \cup \{x_4\} = B \cup \{x_1, x_3, x_2, x_4 \}$ is a subset of $\aff(C)$ with basis $B$ of extended type $5 \hyph 5 \hyph 5 \hyph 5 \hyph (2,3,3,3,3,3)$, using the order $x_1, x_3, x_2, x_4$ of dependent elements. By Theorem~\ref{two12caps-v2}, it follows that $C'$ is a $12$-cap.
    Therefore the cap $C$ can be extended to a larger cap of the same dimension, so it cannot be complete.

    \bigskip

    Now suppose that $\dim(C) > 7$. Note that the first quad closure of $C$ has at most
    \[
        |C| + \binom{|C|}{3}
        = 11 + \binom{11}{3}
        = 176
    \]
    elements. But $\aff(C)$ has $2^{\dim(C)} \geq 2^8 = 256$ elements, so it is impossible for the first quad closure of $C$ to equal $\aff(C)$.
\end{proof}

\bigskip

\begin{proposition} \label{just5555-233333}
    Every $12$-cap of dimension $7$ has a basis of extended type $5 \hyph 5 \hyph 5 \hyph 5 \hyph (2,3,3,3,3,3)$.
\end{proposition}

\begin{proof}
    Let $C$ be an $12$-cap of dimension $7$. By Theorem~\ref{two12caps-v2}, we know $C$ has a basis of extended type $5 \hyph 5 \hyph 5 \hyph 5 \hyph (2,3,3,3,3,3)$ or $5 \hyph 5 \hyph 5 \hyph 5 \hyph (2,3,3,3,3,2)$. If the former, then the proof is complete.

    Suppose $C$ has a basis $B$ of extended type $5 \hyph 5 \hyph 5 \hyph 5 \hyph (2,3,3,3,3,2)$ 
    From the proof of Theorem~\ref{two12caps-v2}, we can write $B = \{a_1, \ldots, a_8\}$ such that have
    \begin{align*}
        x_1 &= a_1 + \cdots + a_5, \\
        x_2 &= a_1 + a_2 + a_3 + a_7 + a_8, \\
        x_3 &= a_3 + \cdots + a_7, \\
        x_4 &= a_2 + a_4 + a_6 + a_7 + a_8,
    \end{align*}
    where $x_1,x_2,x_3,x_4$ are the dependent elements.
    Since $a_5 \in B_1 \cap B_3$ but $a_5 \notin B_2 \cup B_4$, by the Affine Basis Exchange Theorem (\ref{newcapbasis}), we know that $B' = (B \setminus \{a_5\}) \cup \{x_2\}$ is a basis for $C$ with dependent set $D' = (D \setminus \{x_2\}) \cup \{a_5\}$, and
    \begin{align*}
        B'_{a_5} &= \{a_3, a_4, a_6, a_7, x_2\}, \\
        B'_2 &= B_2 = \{a_1, a_2, a_3, a_7, a_8\}, \\
        B'_1 &= \{ a_1, a_2, a_6, a_7, x_2\}, \\
        B'_4 &= B_4 = \{ a_2, a_4, a_6, a_7, a_8\}.
    \end{align*}
    Observe that
    \begin{align*} 
        |B'_{a_5} \cap B'_2| &= 2,
            & |B'_{a_5} \cap B_1| &= 3, \\
        |B'_{a_5} \cap B'_4| &= 3,
            & |B'_2 \cap B'_1| &= 3, \\
        |B'_2 \cap B'_4| &= 3,
            & |B'_1 \cap B'_4| &= 3.
    \end{align*} 
    Thus $B'$ is a basis for $C$ of extended type $5 \hyph 5 \hyph 5 \hyph 5 \hyph (2,3,3,3,3,3)$, with respect to the ordering  of $D' = \{a_5, x_2, x_1, x_4\}$.
\end{proof}

\bigskip

\begin{corollary} \label{one12cap}
    All $12$-caps of dimension $7$ are equivalent.
\end{corollary}

\begin{proof}
    By Proposition~\ref{just5555-233333}, we know every $11$-cap of dimension $7$ has a basis of extended type $5 \hyph 5 \hyph 5 \hyph 5 \hyph (2,3,3,3,3,3)$. In the proof of Theorem~\ref{two12caps-v2}, we constructed a dependent set template for all such caps and bases. By Lemma~\ref{same-template}, this means all such caps are affinely equivalent. Hence there is only one equivalence class of $12$-caps in $\Z_2^7$.
\end{proof}

\section{Caps of size 13}

In this last section, we prove that there are no $13$-caps in $\Z_2^7$. We do this by first showing that all bases for such a cap must have a particular extended type, and then showing that this extended type is actually impossible.


\bigskip

\begin{theorem} \label{13no7}
    Let $C$ be a $13$-cap of dimension $7$ with basis $B$. Then $B$ must have type $5 \hyph 5 \hyph 5 \hyph 5 \hyph 5$, and $|B_{ij}| = 2$ for two disjoint pairs $\{i,j\} \subseteq \{1,\ldots,5\}$, and $|B_{ij}| = 3$ for all other pairs.
\end{theorem}

\begin{proof}
    Let $C$ and $B$ be as in the theorem statement. Since $|B| = 8$, by Lemma~\ref{morethan3} each element of the dependent set $D$ is a sum of five or seven basis elements, and by Lemma~\ref{oneseven}, at most one of is a sum of seven basis elements. Therefore $B$ must have type $5 \hyph 5 \hyph 5 \hyph 5 \hyph 5$ or $7 \hyph 5 \hyph 5 \hyph 5\hyph 5$.
    
    Suppose $B$ has type $7 \hyph 5 \hyph 5 \hyph 5 \hyph 5$. Observe that for any pair $\{j,k\} \subseteq \{2,3,4,5\}$, the set $C_{1jk} = B \cup \{x_1, x_j, x_k\}$ is an $11$-cap of dimension $7$ with basis type $7 \hyph 5 \hyph 5$. By Lemma~\ref{755implies443}, the extended type of such a cap must be $7 \hyph 5 \hyph 5 \hyph (4,3,3)$. This implies that
    \[
        |B_{23}| = |B_{34}| = |B_{45}| = 3.
    \]
    Then $C_{2345} = B \cup \{x_2, x_3, x_4, x_5\}$ is a $12$-cap of dimension $7$ with basis $B$ of extended type $5 \hyph 5 \hyph 5 \hyph 5 \hyph (3,3,3,3,3,3)$. By Lemma~\ref{forbiddenquadruples}, this is impossible. Therefore the basis $B$ must have type $5 \hyph 5 \hyph 5 \hyph 5 \hyph 5$.

    Now consider the numbers $|B_{ij}|$ for pairs $\{i,j\}$ in $\{1,\ldots,5\}$. By Lemma~\ref{twoorthree} we know that each of these numbers is $2$ or $3$. If none or exactly one of these numbers is $2$, then for at least one value $1 \leq i \leq 5$, the set $C'_i = C \setminus \{x_i\}$ is a $12$-cap of dimension $7$ with basis $B$ of extended type $5 \hyph 5 \hyph 5 \hyph 5 \hyph (3,3,3,3,3,3)$. By Lemma~\ref{forbiddenquadruples}, this is impossible, so it must be that at least two of these numbers are $2$.
    
    Applying Proposition~\ref{two12caps} to each $12$-cap $C'_i = C \setminus \{x_i\}$, we conclude that if $|B_{ij}| = |B_{k\ell}| = 2$ for distinct pairs $\{i,j\}$ and $\{k,\ell\}$, then these pairs must be disjoint. Since there are only five elements of $D$, we must have exactly two disjoint pairs $\{i,j\}$ and $\{k,\ell\}$ such that $|B_{ij}| = |B_{k\ell}| = 2$.
\end{proof}

Now we prove that no $13$-cap can exist in dimension $7$.

\bigskip 

\begin{theorem} \label{no13caps}
    There are no $7$-dimensional $13$-caps.   
\end{theorem}

\begin{proof}
    Let $C$ be a $13$-cap of dimension $7$. By Theorem~\ref{13no7}, we know $C$ has a basis $B$ of type $5 \hyph 5 \hyph 5 \hyph 5 \hyph 5$, where $|B_{ij}| = 2$ for two disjoint pairs $\{i,j\}$, and $|B_{ij}| = 3$ for all other pairs. We know $|B_i| = 5$ for each $i$,  and of the ten numbers $|B_{ij}|$ exactly two are $2$ and the rest are $3$. Hence
    \[
        \sum |B_i| = 5(5) = 25
        \quad \text{and} \quad
        \sum |B_{ij}| = 2(2) + 8(3) = 28.
    \]
    Suppose without loss of generality that $|B_{12}| = |B_{34}| = 2$.
    
    For any triple $\{i,j,k\}$, we may apply Corollary~\ref{triples} to the $11$-cap $C_{ijk} = B \cup \{x_i, x_j, x_k\}$ to find that
    \[
        |B_{ijk}| =
        \begin{cases}
            1 & \text{if one of $|B_{ij}|, |B_{ik}|, |B_{jk}|$ is $2$;}\\ 
            2 & \text{if $|B_{ij}| = |B_{ik}| = |B_{jk}| = 3$.}
        \end{cases}
    \]
    There are $\binom{5}{3} = 10$ triples of indices, six of which contain one of the pairs $\{1,2\}$ or $\{3,4\}$, and four of which that do not. Thus, we have
    \[
        \sum |B_{ijk}| = 6(1) + 4(2) = 14.
    \]

    For any quadruple $\{i,j,k,\ell\}$, we may apply Corollary~\ref{quads} to the $12$-cap $C_{ijk\ell} = B \cup \{x_i, x_j, x_k, x_{\ell}\}$ to find that
    \begin{equation} \label{eq:ijkl}
        |B_{ijk\ell}| =
        \begin{cases}
            0  & \text{if $|B_{rs}| = 2$ for two pairs $\{r,s\} \subseteq \{i,j,k,\ell\}$;}\\ 
            1  & \text{if $|B_{rs}| = 2$ for exactly one pair $\{r,s\} \subseteq \{i,j,k,\ell\}$.}
        \end{cases}
    \end{equation}
    There are $\binom{5}{4} = 5$ quadruples of indices. The only quadruple that satisfies the first case of equation~\eqref{eq:ijkl} is $\{1,2,3,4\}$, since $|B_{12}| = |B_{34}| = 2$. The remaining four quadruples satisfy the second case. Hence, we have
    \[
        \sum |B_{ijkl}| = 1(0) + 4(1) = 4.
    \]

    We are now ready to compute $|B_{12345}|$. Using the inclusion-exclusion principle and substituting the values from above, we have
    \begin{align*}
        |B_{12345}|
        &= \left| \bigcup B_i \right|
        - \sum |B_i| + \sum |B_{ij}| - \sum|B_{ijk}|
        + \sum |B_{ijkl}| \\
        &= 8 - 25 + 28 - 14 + 4 \\
        &= 1.
    \end{align*}
    However, $B_{12345} \subseteq B_{1234}$ and $|B_{1234}| = 0$, so we must have $|B_{12345}| = 0$, not $1$. Hence we have a contradiction, so no $7$-dimensional $13$-cap can exist.
\end{proof}

\bigskip

We have the following two corollaries, completing our characterization of caps of dimension 7. 

\bigskip

\begin{corollary} \leavevmode
    \begin{enumerate}
        \item The maximum cap size in $\AG(7,2)$ is $12$.
        \item Every complete cap of dimension $7$ has size $12$.
    \end{enumerate}
\end{corollary}

\begin{proof}
    Since $12$-caps exist in $\AG(7,2)$ and $13$-caps do not, clearly $12$ is the maximum cap size of dimension $7$.    

    By Corollary~\ref{11notcomplete}, we know that $11$-caps cannot be complete. Since $12$-caps are of maximum size in dimension $7$, it follows that they are the only possible complete caps of dimension $7$.
\end{proof}

\printbibliography

\end{document}